\documentclass{amsart}            

\begin{document}
\title[Lie groups in two main classes]
{Lie groups as 3-dimensional almost contact B-metric manifolds
in two main classes}

\author{Miroslava Ivanova}

\address[Miroslava Ivanova]{Department of Informatics and Mathematics, Trakia University,
Stara Zagora, 6000, Bulgaria,
E-mail: mivanova@uni-sz.bg}

\newcommand{\ie}{i.e. }
\newcommand{\Id}{\mathrm{Id}}
\newcommand{\X}{\mathfrak{X}}
\newcommand{\W}{\mathcal{W}}
\newcommand{\F}{\mathcal{F}}
\newcommand{\T}{\mathcal{T}}
\newcommand{\LL}{\mathcal{L}}
\newcommand{\TT}{\mathfrak{T}}
\newcommand{\G}{\mathcal{G}}
\newcommand{\I}{\mathcal{I}}
\newcommand{\M}{(M,\allowbreak{}\ff,\allowbreak{}\xi,\allowbreak{}\eta,\allowbreak{}g)}
\newcommand{\Lf}{(G,\ff,\xi,\eta,g)}
\newcommand{\R}{\mathbb{R}}
\newcommand{\N}{\widehat{N}}
\newcommand{\s}{\mathfrak{S}}
\newcommand{\n}{\nabla}
\newcommand{\nn}{\tilde{\nabla}}
\newcommand{\tg}{\tilde{g}}
\newcommand{\ff}{\varphi}
\newcommand{\D}{{\rm d}}
\newcommand{\id}{{\rm id}}
\newcommand{\al}{\alpha}
\newcommand{\bt}{\beta}
\newcommand{\gm}{\gamma}
\newcommand{\dt}{\delta}
\newcommand{\lm}{\lambda}
\newcommand{\ta}{\theta}
\newcommand{\om}{\omega}
\newcommand{\ep}{\varepsilon}
\newcommand{\ea}{\varepsilon_\alpha}
\newcommand{\eb}{\varepsilon_\beta}
\newcommand{\eg}{\varepsilon_\gamma}
\newcommand{\sx}{\mathop{\mathfrak{S}}\limits_{x,y,z}}
\newcommand{\norm}[1]{\left\Vert#1\right\Vert ^2}
\newcommand{\nf}{\norm{\n \ff}}
\newcommand{\nN}{\norm{N}}
\newcommand{\Span}{\mathrm{span}}
\newcommand{\grad}{\mathrm{grad}}
\newcommand{\thmref}[1]{The\-o\-rem~\ref{#1}}
\newcommand{\propref}[1]{Pro\-po\-si\-ti\-on~\ref{#1}}
\newcommand{\secref}[1]{\S\ref{#1}}
\newcommand{\lemref}[1]{Lem\-ma~\ref{#1}}
\newcommand{\dfnref}[1]{De\-fi\-ni\-ti\-on~\ref{#1}}
\newcommand{\corref}[1]{Corollary~\ref{#1}}



\numberwithin{equation}{section}
\newtheorem{thm}{Theorem}[section]
\newtheorem{lem}[thm]{Lemma}
\newtheorem{prop}[thm]{Proposition}
\newtheorem{cor}[thm]{Corollary}
\newtheorem{defn}{Definition}[section]

\hyphenation{Her-mi-ti-an ma-ni-fold ah-ler-ian}

\begin{abstract}
Three-dimensional almost contact B-metric manifolds
are constructed by a three-parametric family of Lie groups. It is
established the class of the investigated manifolds
which has an important geometrical interpretation.
It is determined also the type of the constructed Lie algebras
in the Bianchi classification.
There are given some geometric characteristics and properties of the considered manifolds.
\end{abstract}


\maketitle

\section*{Introduction}

The geometry of the almost contact B-metric manifolds is the
odd-dimensional counterpart of the geometry of the almost complex
manifolds with Norden metric \cite{GaBo, GriMekDjel}. The
almost contact B-metric manifolds are
introduced in \cite{GaMiGr}. These manifolds are investigated and
studied for instance in \cite{GaMiGr, Man4, Man30, Man-Gri1, Man-Gri2, ManIv38,
ManIv36, NakGri}.

In this paper we focus our attention on the Lie groups considered as
three-dimensional almost contact B-metric manifolds, bearing in mind
the investigations
in \cite{H.Man-Mek}. We construct a family of Lie groups and equip them with
an almost contact B-metric structure from the direct sum of two of main classes.
These classes are $\F_1$ and $\F_{11}$, where the fundamental tensor $F$ is expressed explicitly by the structure $(\ff,\xi,\eta,g)$. The former class is the odd-dimensional analogue of the class of conformal K\"ahler manifolds with Norden metric and the latter class is the only basic class where the Lee form $\om$ is not vanish whereas $\ta$ and $\ta^*$ are zero. Moreover, the metric $g$ is represented on the $\F_1$-manifolds by its horizontal component $g(\ff^2\cdot,\ff^2\cdot)$, i.e. the restriction of $g$ on the contact distribution $\ker(\eta)$, whereas on the $\F_{11}$-manifolds $g$ is represented by its vertical component $\eta\otimes\eta$.

Our aim is to study some important geometric characteristics and
properties of the obtained manifolds.

The paper is organized as follows. In Sec.~\ref{mi:sec2}, we give
necessary facts about the considered manifolds. In Sec.~\ref{mi:sec3},
we construct a family of 3-dimensional Lie groups with almost contact B-metric
structure and characterize it.

\section{Almost contact manifolds with B-metric}\label{mi:sec2}

Let $(M,\ff,\xi,\eta,g)$ be a $(2n+1)$-dimensional \emph{almost
contact B-met\-ric manifold}, i.e. $(\ff,\xi,\eta)$ is an almost
contact structure determined by a tensor field $\ff$ of type
(1,1), a vector field $\xi$ and its dual 1-form $\eta$ as follows:
\begin{equation*}
\ff\xi=0, \quad \ff^2 = -\Id + \eta \otimes \xi, \quad
\eta\circ\ff=0, \quad \eta(\xi)=1,
\end{equation*}
where $\Id$ is the identity; moreover, $g$ is a pseudo-Riemannian
metric satisfying
\begin{equation*}
g(\ff x,\ff y)=-g(x,y)+\eta(x)\eta(y)
\end{equation*}
for arbitrary tangent vectors $x,y\in T_pM$ at an arbitrary point
$p\in M$ \cite{GaMiGr}.

Further, $x$, $y$, $z$, $w$ will stand for arbitrary elements of
$\X(M)$ or $T_pM$ at arbitrary $p\in M$.

We note that the restriction of a B-metric on the contact
distribution $H=\ker(\eta)$ coincides with the corresponding Norden
metric with respect to the almost complex structure; the
restriction of $\ff$ on $H$ acts as an anti-isometry on the
metric on $H$, which is the restriction of $g$ on $H$. Thus, it is obtained
a correlation between a $(2n+1)$-dimensional almost contact
B-metric manifold and a $2n$-dimensional almost complex manifold
with Norden metric. 

The associated metric $\tilde{g}$ of $g$ on $M$ is defined by \(
\tilde{g}(x,y)=g(x,\ff y)+\eta(x)\eta(y)\) and is also a B-metric.
The manifold $(M,\ff,\xi,\eta,\tilde{g})$ is an almost contact
B-metric manifold, too. Both metrics $g$ and $\tilde{g}$ are
indefinite
of signature $(n+1,n)$.

The structure group of $\M$ is $\G\times\I$, where $\G$ is the group $\mathcal{GL}(n;\mathbb{C})\cap
\mathcal{O}(n,n)$ and $\I$ is the
identity on $\Span(\xi)$.

Let $\n$ be the Levi-Civita connection of $g$. The tensor $F$ of
type (0,3) on $M$ is defined by $F(x,y,z)=g\bigl( \left( \n_x \ff
\right)y,z\bigr)$. It satisfies the following identities:
\begin{equation*}\label{mi:F-prop}
F(x,y,z)=F(x,z,y)
=F(x,\ff y,\ff z)+\eta(y)F(x,\xi,z)
+\eta(z)F(x,y,\xi).
\end{equation*}

Almost contact B-metric manifolds were classified in \cite{GaMiGr} into eleven
basic classes $\F_i$
$(i=1,\allowbreak{}2,\allowbreak{}\dots,\allowbreak{}{11})$ with
respect to $F$. These eleven basic classes
intersect in the special class $\F_0$ determined by the condition
$F(x,y,z)=0$. Hence $\F_0$ is the class of almost contact B-metric
manifolds with $\n$-parallel structures, \ie
$\n\ff=\n\xi=\n\eta=\n g=\n \tilde{g}=0$.

Let $\{e_i\}_{i=1}^{2n+1}=\{e_1,e_2,\dots,e_{2n+1}\}$ be a basis
of the tangent space $T_pM$ of $M$ at an arbitrary point $p\in M$.
In this basis, let $g_{ij}$ and $g^{ij}$ be the components of the matrix of $g$ and its inverse, respectively. The Lee forms
$\ta$, $\ta^*$ and $\om$ associated with $F$ are defined by:
\begin{equation*}\label{mi:Lee forms}
\ta(z)=g^{ij}F(e_i,e_j,z), \quad \ta^*(z)=g^{ij}F(e_i,\ff e_j,z),
\quad \om(z)=F(\xi,\xi,z).
\end{equation*}

The square norms of $\nabla \ff$, $\n \eta$ and $\n \xi$ are
defined by:
\begin{equation}
\begin{array}{c}\label{mi:sn f-eta-xi}
\norm{\nabla \ff}=g^{ij}g^{ks} g\bigl(\left(\nabla_{e_i}
\ff\right)e_k,\left(\nabla_{e_j} \ff\right)e_s\bigr),\\[4pt]
\norm{\n\eta}=g^{ij}g^{ks}\left(\nabla_{e_i}\eta\right)e_k\left(\nabla_{e_j}
    \eta\right)e_s, \quad \norm{\n\xi}=g^{ij}g\left(\nabla_{e_i} \xi,\nabla_{e_j}
    \xi\right),
\end{array}
\end{equation}
respectively (\cite{Man33, Man30}). Let us remark, if $\M$ is an $\F_0$-manifold then the 
equality $\norm{\nabla \ff}=0$ is valid, but the inverse
implication is not always true. According to \cite{Man30}, an
almost contact B-metric manifold satisfying the condition
$\norm{\nabla \ff}=0$ is called an
\emph{isotropic-$\F_0$-manifold}.

The curvature tensor $R$ of type (1,3) of $\n$ is defined by
$R(x,y)z=\n_x\n_yz-\n_y\n_xz-\n_{[x,y]}z$. The corresponding
tensor of $R$ of type (0,4) is denoted by the same letter and is
defined by  $R(x,y,z,w)=g(R(x,y)z,w)$.

The Ricci tensor $\rho$ and the scalar curvature $\tau$ for $R$ as
well as their associated quantities are defined as usual by
$\rho(x,y)=g^{ij}R(e_i,x,y,e_j)$, $\tau=g^{ij}\rho(e_i,e_j)$,
$\rho^{*}(x,y)=g^{ij}R(e_i,x,y,\ff e_j)$ and
$\tau^{*}=g^{ij}\rho^{*}(e_i,e_j)$, respectively.

In \cite{ManNak1}, an almost contact B-metric manifold $\M$ is
called \emph{$*$-Einstein} if the Ricci tensor has the form
$\rho=\al \tg|_H$, where $\al$ is a real constant.

If $\al$ is a non-degenerate 2-plane (section) in $T_pM$, $p\in M$,
then, according to \cite{NakGri}, the special
2-planes with respect to $(\ff,\xi,\eta,g)$ are: a \emph{totally
real section} if $\al$ is orthogonal to its $\ff$-image $\ff\al$
and $\xi$, a \emph{$\ff$-holomorphic  section} if $\al$ coincides
with $\ff\al$ and a \emph{$\xi$-section} if $\xi$ lies on $\al$.

The sectional curvature $k(\al; p)(R)$ for $R$ of $\al$ having an arbitrary
basis $\{x,y\}$ at $p$ is given by
\begin{equation}\label{mi:sec curv}
k(\al; p)(R)=\frac{R(x,y,y,x)}{g(x,x)g(y,y)-g(x,y)^2}.%
\end{equation}

In \cite{Man30}, a linear connection $D$ is called a \emph{natural
connection} on an arbitrary almost contact B-metric manifold if
its structure tensors are parallel with respect to $D$, \ie
$D\ff=D\xi=D\eta=Dg=D\tilde{g}=0$.
According to \cite{ManIv36}, $D$ is
natural on $(M,\ff,\allowbreak\xi,\eta,g)$ if and only if
$D\ff=Dg=0$.
Natural connections exist on any almost contact B-metric manifold. They are restricted to $\n$ only on an $\F_0$-manifold.

In \cite{Man-Gri2}, it is introduced a natural connection
$\dot{D}$ on $\M$ in all basic classes by
\begin{equation}\label{mi:defn-fiB}
\begin{array}{l}
\dot{D}_xy=\n_xy+\frac{1}{2}\bigl\{\left(\n_x\ff\right)\ff
y+\left(\n_x\eta\right)y\cdot\xi\bigr\}-\eta(y)\n_x\xi.
\end{array}
\end{equation}
This connection is called a \emph{$\ff$B-connection} in
\cite{ManIv37} and it is studied for the main classes
$\F_1,\F_4,\F_5,\F_{11}$  in \cite{Man-Gri2, Man3, Man4}. The
$\ff$B-connection is the odd-dimensional counterpart of the
B-connection on the corresponding almost complex manifold with
Norden metric, studied for the class $\W_1$ in \cite{GaGrMi}.

In \cite{ManIv38}, a natural connection $\ddot{D}$ is
called a \emph{$\ff$-canonical connection} on
$(M,\ff,\xi,\allowbreak\eta,g)$ if its torsion tensor $\ddot{T}$,
defined by $\ddot T(x,y,z)=g(\ddot D_xy-\ddot D_yx-[x,y],z)$,
satisfies the following identity:
\begin{equation*}\label{mi:T-can}
\begin{split}
    &\ddot{T}(x,y,z)-\ddot{T}(x,z,y)-\ddot{T}(x,\ff y,\ff z)
    +\ddot{T}(x,\ff z,\ff y)=\\
    &=\eta(x)\left\{\ddot{T}(\xi,y,z)-\ddot{T}(\xi,z,y)
    -\ddot{T}(\xi,\ff y,\ff z)+\ddot{T}(\xi,\ff z,\ff y)\right\}\\
    &+\eta(y)\left\{\ddot{T}(x,\xi,z)-\ddot{T}(x,z,\xi)
    -\eta(x)\ddot{T}(z,\xi,\xi)\right\}\\
    &-\eta(z)\left\{\ddot{T}
(x,\xi,y)-\ddot{T}(x,y,\xi)-\eta(x)\ddot{T}(y,\xi,\xi)\right\}.
\end{split}
\end{equation*}
Moreover, there is proven that the $\ff$B-connection and the
$\ff$-canonical connection coincide on an almost contact B-metric
manifold if and only if the manifold belongs to the class
$\F_1\oplus\F_2\oplus\F_4\oplus\F_5\oplus\F_6\oplus\F_8\oplus\F_9\oplus\F_{10}\oplus\F_{11}$.

Another natural connection on $\M$, which is called $\ff$KT-connection, is introduced and studied in \cite{Man30}. This connection is defined by the condition its torsion to be a 3-form. There is proven that the $\ff$KT-connection exists if and only if $\M$ belongs to $\F_3\oplus\F_7$.

According to \cite{H.Man1}, the class of 3-dimensional almost contact B-metric manifolds is
$
\F_1\oplus\F_4\oplus\F_5\oplus\F_8\oplus\F_9\oplus\F_{10}\oplus\F_{11}.
$

\section{A family of Lie groups as 3-dimensional $(\F_1\oplus\F_{11})$-manifolds}\label{mi:sec3}

Let $G$ be a three-dimensional real connected Lie group and
$\mathfrak{g}$ be its Lie algebra. Let
$\left\{e_0,e_1,e_2\right\}$ be a global basis of left-invariant
vector fields on $G$. We define an almost contact structure
$(\ff,\xi,\eta)$ on $G$ by
\begin{equation}\label{mi:f}
\begin{array}{l}
\ff e_0 = o,\quad \ff e_1 = e_2,\quad \ff e_2 =-e_1,\quad \xi = e_0;\\
\eta(e_0)=1,\quad \eta(e_1)=\eta(e_2)=0,
\end{array}
\end{equation}
where $o$ is the zero vector field.
We define a B-metric $g$ on $G$ by
\begin{equation}\label{mi:g}
\begin{array}{l}
g(e_0,e_0)=g(e_1,e_1)=-g(e_2,e_2)=1,
\\
g(e_0,e_1)=g(e_0,e_2)=g(e_1,e_2)=0.
\end{array}
\end{equation}

We consider the Lie algebra $\mathfrak{g}$ on $G$, determined by the
following non-zero commutators:
\begin{equation}\label{mi:komutator}
\left[e_0,e_1\right]=-de_0,\quad \left[e_0,e_2\right]=ce_0, \quad
\left[e_1,e_2\right]=ae_1+be_2,
\end{equation}
where $a,b,c,d\in\R$.
We verify immediately that the Jacobi identity for this Lie algebra is
satisfied if and only if the condition $ad=bc$ holds. By virtue of the latter
equality, $G$ is a three-parametric family of Lie groups.

\begin{thm}\label{mm:thm F1+F11}
Let $(G,\ff,\xi,\eta,g)$ be a three-dimensional connected Lie
group with almost contact B-metric structure determined by
\eqref{mi:f}, \eqref{mi:g} and \eqref{mi:komutator} with the condition $ad=bc$. Then it
belongs to the class $\F_1\oplus\F_{11}$.
\end{thm}

\begin{proof}
The well-known Koszul equality of the Levi-Civita connection $\n$ of $g$
\begin{equation}\label{mi:Koszul}
2g\left(\n_{e_i}e_j,e_k\right)=g\left(\left[e_i,e_j\right],e_k\right)
+g\left(\left[e_k,e_i\right],e_j\right)+g\left(\left[e_k,e_j\right],e_i\right)
\end{equation}
implies the following form of the components
$F_{ijk}=F(e_i,e_j,e_k)$ of the tensor $F$:
\begin{equation}\label{mi:F}
\begin{split}
 2F_{ijk}=g\left(\left[e_i,\ff e_j\right]-\ff
\left[e_i,e_j\right],e_k\right)&+g\left(\ff
\left[e_k,e_i\right]-\left[\ff e_k,e_i\right],e_j\right)\\[4 pt]
&+g\left(\left[e_k,\ff e_j\right]-\left[\ff
e_k,e_j\right],e_i\right).
\end{split}
\end{equation}
Using \eqref{mi:F} and \eqref{mi:komutator}, we get the following non-zero
components $F_{ijk}$:
\begin{equation}\label{mi:Fijk}
\begin{array}{ll}
F_{001}=F_{010}=c, \quad &F_{002}=F_{020}=d, \\[4 pt]
F_{111}=F_{122}=2a, \quad &F_{211}=F_{222}=-2b.
\end{array}
\end{equation}

By direct verification, we establish that the components in
\eqref{mi:Fijk} satisfy the condition $F=F^1+F^{11}$ for the components $F^s$
of $F$ in the basic classes $\F_s$ $(s=1,11)$ having the following form (see \cite{H.Man1})
\begin{equation}\label{Fi3}
\begin{array}{l}
F^{1}(x,y,z)=\left(x^1\ta_1-x^2\ta_2\right)\left(y^1z^1+y^2z^2\right),\\[4pt]
\phantom{F^{1}(x,y,z)=}
\ta_1=F_{111}=F_{122},\quad \ta_2=-F_{211}=-F_{222}; \\[4pt]
F^{11}(x,y,z)=x^0\bigl\{\left(y^1z^0+y^0z^1\right)\om_{1}
+\left(y^2z^0+y^0z^2\right)\om_{2}\bigr\},\\[4pt]
\phantom{F^{11}(x,y,z)=}
\om_1=F_{010}=F_{001},\quad \om_2=F_{020}=F_{002},
\end{array}
\end{equation}
where $\ta_i=\ta(e_i)$ and $\om_i=\om(e_i)$ $(i=1,2)$ are determined by $\ta_1=2a$, $\ta_2=2b$, $\om_1=c$, $\om_2=d$.
Therefore, the manifold
$(G,\ff,\xi,\eta,g)$ belongs to the class $\F_1\oplus\F_{11}$ of
the mentioned classification.

Obviously, $(G,\ff,\xi,\eta,g)$ belongs to $\F_1$, $\F_{11}$ and $\F_0$ if and only
if for all $i\in\{1,2\}$ the parameters $\ta_i$, $\om_i$ and $\ta_i=\om_i$ vanish, respectively.

Bearing in mind the above, the commutators in
\eqref{mi:komutator} take the form
\begin{equation}\label{mi:komutator=}
\begin{array}{l}
\left[e_0,e_1\right]=-\om_2e_0,\quad \left[e_0,e_2\right]=\om_1e_0, \quad
\left[e_1,e_2\right]=\frac12(\ta_1e_1+\ta_2e_2),\\[4pt]
 \ta_1\om_2=\ta_2\om_1
\end{array}
\end{equation}
in terms of the basic components of the Lee forms $\ta$ and $\om$.
\end{proof}

In \cite{Kow}, it is consider the matrix Lie group $G_I$ of the
following form
\[
G_I=\left(%
\begin{array}{ccc}
  e^{-z} & 0 & x \\
  0 & e^z & y \\
  0 & 0 & 1 \\
\end{array}%
\right),
\]
where $x,y,z\in\mathbb{R}$. This Lie group is the group of
hyperbolic motions of the plane $\mathbb{R}^2$. The corresponding
Lie algebra is determined by the following commutators:
\begin{equation*}
\mathfrak{g}_I: \quad \left[X_1,X_3\right]=X_1,\quad
\left[X_2,X_3\right]=-X_2, \quad \left[X_1,X_2\right]=0.
\end{equation*}
The type of this Lie algebra according to the Bianchi
classification given in \cite{Bia1, Bia2} of three-dimensional
real Lie algebras is Bia(V).

Bearing in mind \cite{H.Man2}, the matrix representation of the
Lie algebra $\mathfrak{g}_I$ in the class $\F_1\oplus\F_{11}$ is
\[
A=\left(%
\begin{array}{ccc}
  -\om_2 r+\om_1 s & 0 & 0 \\
  \om_2 t & \frac{1}{2}\ta_1 q & \frac{1}{2}\ta_2 q \\
  -\om_1 t & -\frac{1}{2}\ta_1 p & -\frac{1}{2}\ta_2 p \\
\end{array}%
\right),
\]
where $p,q,r,s,t\in\mathbb{R}$. Then substituting $X_1=e_2$,
$X_2=e_0$, $X_3=-e_1$, $\ta_1=0$, $\ta_2=2$, $\om_1=0$, $\om_2=-1$
or $X_1=e_1$, $X_2=-e_0$, $X_3=e_2$, $\ta_1=2$, $\ta_2=0$,
$\om_1=-1$, $\om_2=0$, we have that $G_I$ can be considered as an
almost contact B-metric manifold of the class $\F_1\oplus\F_{11}$
\cite{H.Man3}.

Using \eqref{mi:Koszul} and \eqref{mi:komutator=}, we obtain the
components of $\n$:
\begin{equation}\label{mi:nabli}
\begin{array}{c}
\begin{array}{c}
\n_{e_0}e_0=\om_2e_1+\om_1e_2, \quad \n_{e_0}e_1=-\om_2e_0, \quad \n_{e_0}e_2=\om_1e_0,
\end{array}
\\[4pt]
\begin{array}{ll}
\n_{e_1}e_1=\frac{1}{2}\ta_1e_2, \quad & \n_{e_1}e_2=\frac{1}{2}\ta_1e_1, \\[4pt]
\n_{e_2}e_1=-\frac{1}{2}\ta_2e_2,\quad & \n_{e_2}e_2=-\frac{1}{2}\ta_2e_1.
\end{array}
\end{array}
\end{equation}

We denote by $R_{ijkl}=R(e_i,e_j,e_k,e_l)$ the components of the
curvature tensor $R$, $\rho_{jk}=\rho(e_j,e_k)$ of the Ricci
tensor $\rho$, $\rho^*_{jk}=\rho^*(e_j,e_k)$ of the associated
Ricci tensor $\rho^*$ and $k_{ij}$ of the sectional curvature for
$\n$ of the basic 2-plane $\al_{ij}$ with a basis $\{e_i,e_j\}$,
where $e_i, e_j\in\{e_0, e_1, e_2\}$. On the considered manifold
$(G,\ff,\xi,\eta,g)$ the basic 2-planes $\al_{ij}$ are:
$\ff$-holomorphic section
--- $\al_{12}$ and $\xi$-sections --- $\al_{01}$, $\al_{02}$. Further, by \eqref{mi:sec curv}, \eqref{mi:g}, \eqref{mi:komutator=} and
\eqref{mi:nabli}, we compute
\begin{equation}\label{mi:R-F1+F11}
\begin{array}{c}
R_{0110}=k_{01}=\frac{1}{2}\ta_1\om_1-\om_2^2, \quad R_{0220}=-k_{02}=-\om_1^2+\frac{1}{2}\ta_2\om_2,\\[4 pt]
R_{0120}=\rho_{12}=\frac{1}{2}\rho^*_{00}=\frac{1}{2}\tau^*=\left(\om_1-\frac{1}{2}\ta_1\right)\om_2,\\[4 pt]
R_{1221}=-k_{12}=-\rho^*_{12}=-\frac{1}{4}(\ta_1^2-\ta_2^2),\\[4pt]
\rho_{00}=\frac12\left(\ta_1\om_1-\ta_2\om_2\right)+(\om_1^2-\om_2^2), \\[4pt] \rho_{11}=\frac14\left(\ta_1^2-\ta_2^2\right)+\frac12\ta_1\om_1-\om_2^2,\quad
\rho_{22}=-\frac14\left(\ta_1^2-\ta_2^2\right)+\frac12\ta_2\om_2-\om_1^2,\\[4pt]
\tau=\frac12\left(\ta_1^2-\ta_2^2\right)+\left(\ta_1\om_1-\ta_2\om_2\right)+2(\om_1^2-\om_2^2).
\end{array}
\end{equation}
The rest of non-zero components of $R$, $\rho$ and $\rho^*$ are
determined by \eqref{mi:R-F1+F11} and the properties
$R_{ijkl}=R_{klij},$ $R_{ijkl}=-R_{jikl}=-R_{ijlk},$
$\rho_{jk}=\rho_{kj}$ and $\rho^*_{jk}=\rho^*_{kj}.$

Taking into account \eqref{mi:sn f-eta-xi}, we have
\begin{gather}
\norm{\nabla
\ff}=2\left(\ta_1^2-\ta_2^2+\om_1^2-\om_2^2\right),\label{norm fi}
\\[4pt]
\norm{\nabla \eta}=\norm{\nabla \xi}=-\left(\om_1^2-\om_2^2\right).\label{mi:norm n}
\end{gather}

Denoting $\ep=\pm 1$, the following proposition is valid.

\begin{prop}
The following characteristics are valid for $(G,\ff,\xi,\eta,g)$:
\begin{enumerate}
\item The $\ff$B-connection $\dot D$ (respectively,
$\ff$-canonical connection $\ddot D$) is zero in the basis $\{e_0,e_1,e_2\}$;
\item The manifold is flat if and only if it belongs to either a subclass of $\F_1\oplus\F_{11}$ determined by conditions $\ta_1=\ep\ta_2=2\om_1=2\ep\om_2$ or a subclass of $\F_1$ determined by $\ta_1=\ep\ta_2$;
\item The manifold is Ricci-flat
(respectively, $*$-Ricci-flat) if and only if
it is flat.
\item 
The manifold is scalar flat if and only
if the conditions $\ta_1\pm\ta_2=\om_1\pm\om_2=0$ are satisfied;
\item 
The manifold is $*$-scalar flat if and only
if one of the following four conditions are satisfied: \\
$\ta_1-2\om_1=\ta_2-2\om_2=0$, \quad
$\ta_1=\om_1=0$, \quad
$\ta_2=\om_2=0$, \quad
$\om_1=\om_2=0$;
\item 
The manifold is an isotropic-$\F_0$-manifold if and only
if the conditions $\ta_1\pm\ta_2=\om_1\pm\om_2=0$ are satisfied.
\end{enumerate}
\end{prop}

\begin{proof}
Using \eqref{mi:defn-fiB}, \eqref{mi:f} and \eqref{mi:nabli}, we
get immediately the assertion (1).
%
%
The assertions (2), (4) and (5) are obtained using \eqref{mi:R-F1+F11}.
On the three-dimensional almost contact B-metric manifold with the
basis $\left\{e_0,e_1,e_2\right\}$, bearing in mind the definition
equalities of the Ricci tensor $\rho$ and the $\rho^*$, we have
\[
\rho_{jk}=R_{0jk0}+R_{1jk1}-R_{2jk2} \qquad
\rho^*_{jk}=R_{1kj2}+R_{2jk1}.
\]
By virtue of the latter equalities, we get the assertion (3).
Bearing in mind \eqref{norm fi}, we establish the truthfulness of (6).
\end{proof}

According to \eqref{mi:R-F1+F11}, \eqref{norm fi} and
\eqref{mi:norm n}, we obtain
\begin{prop}\label{mm:prop}
The following properties are equivalent for the manifold $(G,\ff,\xi,\eta,g)$:
\begin{enumerate}
\item it is an
isotropic-$\F_0$-manifold;
\item it is scalar flat;
\item it is $*$-Einsteinian;
\item
the dual vectors $\Theta$, $\Omega$ of $\ta$, $\om$, respectively, and the vector $\n_{\xi}\xi$ 
are isotropic;
\item the sectional curvatures $k_{ij}$ vanish;
\item the equalities $\ta_1\pm\ta_2=\om_1\pm\om_2=0$ are
valid.
\end{enumerate}
\end{prop}


\end{document}